\documentclass{amsart}

\usepackage{amsmath, amsthm, mathtools, amssymb, tikz-cd, hyperref}

\newtheorem{thm}{Theorem}[section]
\newtheorem*{thmA}{\textbf{Theorem A}}
\newtheorem*{thmB}{\textbf{Theorem B}}
\newtheorem*{thmC}{\textbf{Theorem C}}
\newtheorem*{corD}{\textbf{Corollary D}}

\newtheorem{cor}[thm]{Corollary}
\newtheorem{lem}[thm]{Lemma}
\newtheorem{prop}[thm]{Proposition}

\theoremstyle{definition}
\newtheorem{defn}[thm]{Definition}

\newtheorem{cons}[thm]{Construction}

\newtheorem{exmpl}[thm]{Example}

\newtheorem{rem}[thm]{Remark}

\newcommand{\Cref}[1]{{Corollary~\ref{#1}}}
\newcommand{\Csref}[1]{{Construction~\ref{#1}}}
\newcommand{\Eref}[1]{{Example~\ref{#1}}}
\newcommand{\Lref}[1]{{Lemma~\ref{#1}}}
\newcommand{\Pref}[1]{{Proposition~\ref{#1}}}
\newcommand{\Rref}[1]{{Remark~\ref{#1}}}
\newcommand{\Sref}[1]{{Section~\ref{#1}}}
\newcommand{\Tref}[1]{{Theorem~\ref{#1}}}

\newcommand{\sub}{\subseteq}
\newcommand{\sg}[1]{\left\langle #1 \right\rangle}
\newcommand{\set}[1]{\left\{ #1 \right\}}
\renewcommand{\Pr}{\mathbb{P}}

\newcommand{\E}{\mathbb{E}}
\newcommand{\N}{\mathbb{N}}
\newcommand{\Sm}{\mathcal{S}_{\vec{m}}}
\newcommand{\SSm}{\overline{\mathcal{S}}_{\vec{m}}}
\newcommand{\ind}[1]{\boldsymbol{1}_{#1}}
\newcommand{\ff}{F}

\newcommand{\floor}[1]{\left\lfloor #1\right\rfloor}
\newcommand{\ceil}[1]{\left\lceil #1\right\rceil}
\newcommand{\dist}{\operatorname{dist}}
\newcommand{\Cay}{\operatorname{Cay}}
\newcommand{\eps}{\varepsilon}

\title{The speed of random walks on semigroups}

\author{Guy Blachar}
\address{Department of Mathematics, The Weizmann Institute of Science, Rehovot 7610001, Israel}
\email{guy.blachar@gmail.com}

\author{Be'eri Greenfeld}
\address{Department of Mathematics, University of Washington, Seattle, WA, 98195, USA}
\email{grnfld@uw.edu}

\makeatletter
\@namedef{subjclassname@2020}{
	\textup{2020} Mathematics Subject Classification}
\makeatother

\begin{document}

\begin{abstract}
We construct, for each real number $0\leq \alpha \leq 1$, a random walk on a finitely generated semigroup whose speed exponent is $\alpha$. We further show that the speed function of a random walk on a finitely generated semigroup can be arbitrarily slow, yet tending to infinity. These phenomena demonstrate a sharp contrast from the group-theoretic setting.
On the other hand, we show that the distance of a random walk on a finitely generated semigroup from its starting position is infinitely often larger than a non-constant universal lower bound, excluding a certain degenerate case.
\end{abstract}

\maketitle

\section{Introduction}

Random walks have proven to be essential in the study of many algebraic, geometric and probabilistic properties of finitely generated groups, such as amenability, growth, and metric embeddings of groups \cite{ANP,BartholdiVirag05,ErschlerZheng20,Kesten59,NP,Zheng22}. A key ingredient in the study of these properties is to quantify the ``drift'' of the random walk from its starting point, and investigate how properties of the group as those mentioned above reflect in the asymptotic behavior of the random walk.

One such quantity which has been explored in these studies is the \textbf{speed} (also known as the \textbf{rate of escape}) of the random walk. Let $\Gamma$ be a finitely generated group with a finite symmetric generating set $T$, and let $\{R_n\}_{n=0}^{\infty}$ be a finitely supported symmetric random walk on $\Gamma$ such that the support of $R_1$ generates $\Gamma$ as a semigroup. The speed function of $\{R_n\}_{n=0}^{\infty}$ is defined as the function $n\mapsto\E\left|R_n\right|$, where $\left|\cdot\right|$ denotes the word metric on $\Gamma$ with respect to $S$. We further define the \textbf{speed exponent} of the random walk as $\limsup_{n\to\infty}\log_n\E\left|R_n\right|$.

Consider the following two test cases. If $\Gamma=\mathbb{Z}$ then the speed function of any symmetric random walk over it is $\asymp \sqrt{n}$ (the `diffusive' case), as follows from the Central Limit Theorem. In fact, Hebisch--Saloff-Coste \cite{HSC93} proved that for any nilpotent group, $\E|R_n|\asymp \sqrt{n}$ (see \cite{Thompson13} for certain extensions).
On the other extreme, if $\Gamma=\mathbb{F}_d$ is a non-abelian free group then the speed function of any random walk on it is $\asymp n$ (the `ballistic' case). This generalizes to simple random walks on any non-amenable group. Indeed, Kesten's theorem \cite{Kesten59} shows that the spectral radius of such random walks is smaller than $1$, which in particular implies that the speed function grows linearly.

The following ``inverse problem'' thus naturally arises: Which functions can be (asymptotically) realized as speed functions of random walks on finitely generated groups? Specifically, which real numbers occur as speed exponents of random walks on finitely generated groups? These fundamental problems are attributed in \cite{AV} to Vershik (and reformulated by Peres).

The speed function of a random walk on $\Gamma$ is bounded if and only if $\Gamma$ is finite. Lee--Peres \cite{LP} (also attributed there to Erschler and Vir\'{a}g) showed that if $\Gamma$ is infinite then $\E|R_n|\gtrsim \sqrt{n}$; in fact, they proved it for random walks on any amenable transitive graph. Hence the speed exponent of a random walk on any infinite group lies in the interval $[\frac{1}{2},1]$.
Erschler \cite{Erschler01,Erschler06,Erschler10} showed that the numbers $1-2^{-i},\ i=1,2,\dots$ occur as speed exponents of random walks on iterated wreath products of $\mathbb{Z}$; realized functions oscillating between two different speed exponents as speed functions; and constructed sub-linear speed functions that are arbitrarily close to linear. 
Next, Amir--Vir\'{a}g realized any `sufficiently regular' function $f$ such that $n^{3/4} \lesssim  f(n) \lesssim n^{1-\varepsilon}$ for some $\varepsilon>0$ as the speed function over a suitable finitely generated group, extending the known spectrum of speed exponents to $[\frac{3}{4},1]$; see also \cite{Amir17} for a joint control on the speed and entropy functions.
Finally, Brieussel--Zheng \cite{BZ21} realized any sufficiently regular function $\sqrt{n}\lesssim f(n) \lesssim n$ as the speed function over some finitely generated group, thereby identifying the spectrum of possible speed exponents as the full interval $[\frac{1}{2},1]$. For more on random walks on discrete groups see the survey \cite{Zheng22}. \medskip

In this paper, we focus on random walks on \textit{semigroups}. Such random walks have been studied in \cite{ForghaniTiozzo19, GrayKambites20, HognasMukherjea11}. Our main goal in the paper is to prove that the spectrum of speed exponents for random walks on semigroups is the full interval $[0,1]$, in a vast contrast to the group-theoretic setting:

\begin{thmA}[{Arbitrary speed exponents}]
Let $\alpha\in [0,1]$. Then there exists a finitely generated semigroup $S$ and a simple random walk $\{R_n\}_{n=0}^{\infty}$ on $S$ such that $$\limsup_{n\rightarrow \infty} \log_n \E|R_n| = \alpha.$$
\end{thmA}
We also show that, again in contrast with the group-theoretic case, there is no universal gap between constant speed functions and arbitrary increasing function; namely, there is no semigroup analog of the aforementioned result of Lee--Peres:
\begin{thmB}[{Arbitrarily slow speed functions}]\label{thm:no-low-bound}
Let $\omega$ be an arbitrarily slow function such that $\omega(n)\xrightarrow{n\rightarrow \infty} \infty$. Then there exists a finitely generated semigroup $S$ and a simple random walk $\{R_n\}_{n=0}^{\infty}$ on $S$ such that $$\E|R_n|\le \omega(n)\ \ \text{for infinitely many}\ n$$
while $\E|R_n|\xrightarrow{n\rightarrow \infty} \infty$.
\end{thmB}

Although semigroups are more algebraically flexible than groups, there are some fundamental challenges in constructing semigroups whose speed functions are significantly different from those of groups, as we now demonstrate. First, if a semigroup $S$ contains a finite one-sided ideal then the speed function is bounded (\Pref{prop:finite-ideals}). This eliminates many natural constructions of semigroups, including monomial semigroups and generally semigroups with zero, so $S$ must be far from a pathological semigroup and, in some sense `closer' to a group. However, if $S$ surjects onto an infinite group then, by \cite{LP}, its speed function is $\gtrsim \sqrt{n}$. So, in some sense, $S$ must also be `orthogonal' to groups. Moreover, if there is a semigroup homomorphism, $S\rightarrow \mathbb{N}$ with infinite image, then its speed is $\asymp n$. This observation eliminates, in particular, the use of graded semigroups.\medskip

Finally, we show that while the speed functions of random walks on semigroups are quite flexible, the rate at which such random walks escape from their starting points -- namely, the distribution of the random variables $\{|R_n|\}_{n=1}^{\infty}$ itself -- cannot be completely arbitrary. This follows from a more general result on random walks on directed graphs.

Let $G=(V,E)$ be a rooted directed graph, i.e., a directed graph with a root $o\in V$ such that there is a path from $o$ to each vertex in $G$. We define the \textbf{rooted ball spread} of $G$ as the function
\[
    \ff_G(n) = \min_{v\in V} \max_{w:\dist(v,w)\le n}\dist(o,w).
\]
When $G$ is undirected (i.e., $E$ is symmetric), then $\ff_G(n)\ge\frac{n}{2}$ (see \Rref{rem:undirected}). If~$G$ has a finite strongly connected component, then $\ff_G$ is bounded (\Pref{prop:f-bound}).
We assume that all the outdegrees in~$G$ are finite. If the outdegrees of vertices in~$G$ are uniformly bounded then $F_G(n) \gtrsim \log n$, see \Cref{cor:f-bound}.

We now consider a simple random walk $\{R_n\}_{n=0}^{\infty}$ on $G$ starting from the root $o$. 

\begin{thmC}
Let $G=(V,E)$ be a rooted directed graph without finite strongly connected components, where all outdegrees are bounded by $d\in\N$. Let $\{R_n\}_{n=0}^{\infty}$ be a simple random walk on $G$. Then, almost surely, $$\dist(o,R_n)\ge F_G(\log_d n+\log_d\log_d n)$$ for infinitely many values of $n$.
\end{thmC}

For random walks on semigroups, this theorem translates to:

\begin{corD} \label{cor:semigroup loglog}
     Let $S$ be a $d$-generated semigroup with no finite right ideals. Let $\{R_n\}_{n=0}^{\infty}$ be a simple (right) random walk on $S$. Then, almost surely, $$|R_n|\ge \ff_S(\log_d n+\log_d\log_d n)$$ for infinitely many values of $n$.
\end{corD}

We show that this bound is tight in \Eref{exmpl:speed-bounded-dist-not}.
\medskip

\noindent \emph{Structure of the paper.} In Section~\ref{sec2}, we provide necessary background and preliminaries on Cayley graphs, random walks and presentations of semigroups. In Section \ref{sec3} we give the main construction of a semigroup which will serve us in proving Theorem A, which is done in Sections~\ref{sec4} and \ref{sec5}. In Section~\ref{sec6} we modify the main construction to prove Theorem B. In Section~\ref{sec7} we study random walks on directed graphs and conclude with Theorem C and Corollary D. Finally, Section~\ref{sec8} discusses the case of semigroups of bounded speed.

\medskip

\noindent \textit{Conventions.}
We write $f\gtrsim g$ when $f(n)\geq cg(n)$ for some positive constant $c>0$ and $f\asymp g$ if $f\gtrsim g$ and $g\gtrsim f$.
Our semigroups are assumed to have an identity element (namely, they are monoids). Random walks are assumed to be finitely supported and non-degenerate.

\section{Preliminaries on semigroups and random walks}\label{sec2}

\subsection{Cayley graphs}

Let $S$ be a finitely generated semigroup, with a finite generating set $T$. The \textbf{Cayley graph} of $S$ with respect to $T$ is the directed graph $\Cay(S,T)=(V,E)$, with vertex set $V=S$ and edges $E=\set{(a,at)\mid a\in S,t\in T}$. Any Cayley graph is a rooted directed graph, with root $1$; in other words, there is a path from $1$ to any $a\in S$.

While a Cayley graph $\Cay(S,T)$ of a finitely generated semigroup $S$ need not be vertex-transitive, it still carries some symmetries. Indeed, the outdegrees of all vertices is the same, and equals the number of generators $\left|T\right|$. Furthermore, $S$ acts on $\Cay(S,T)$ via multiplication from the left on the vertices. This resembles the action of a group on its Cayley graphs by translations, however here distinct vertices may be mapped to the same vertex. Note, however, that the indegrees of the vertices need not be even finite, and the Cayley graph may contain loops and double edges.

The ideal structure of $S$ can also be read from its Cayley graphs. Right ideals of~$S$ correspond to closed subsets in $\Cay(S,T)$, i.e., sets of vertices with no outgoing edge from the set to its complement. Principal right ideals of $S$ are the sets obtained by taking the minimal strongly connected subsets containing a given vertex. Therefore, strongly connected components of the Cayley graph are principal right ideals (while the reverse claim is not true in general). Left ideals of $S$ correspond to sets of vertices which are closed under the action of $S$ by translations.

\subsection{Random walks}

Let $(G,o)$ be a rooted directed graph which is locally finite, i.e., all outdegrees are finite. A \textbf{simple random walk} on $(G,o)$ is a sequence of random variables $\set{R_n}_{n=0}^{\infty}$, such that $R_0=o$, and $R_n$ is chosen uniformly from the neighbors of $R_{n-1}$, independently between the different steps of the random walk.

We define the \textbf{speed} of a simple random walk $R_n$ on $(G,o)$ as the function $\ell_G(n)=\E[\dist(o,R_n)]$, where $\dist(o,v)$ denotes the distance of the vertex $v\in V$ from the root $o$. Note that $\dist(o,v)<\infty$ for all $v\in V$, hence $\ell_G(n)<\infty$ for all $n$.

Given a finitely generated semigroup $S$ with a finite generating set $T$, a \textbf{simple random walk} on $S$ (with respect to $T$) is a simple random walk on $\Cay(S,T)$. This can be rewritten as a sequence of random variables $R_n=X_1\cdots X_n$, where $X_1,X_2,\dots$ are i.i.d.\ random variables distributed uniformly on $T$. When clear from the context, we do not emphasize the generating set. The distance can then be interpreted as
\[
    \left|a\right|\coloneqq\dist(1,a)=\inf\set{k\ge 0\mid \exists t_1,\dots,t_k\in T:a=t_1\cdots t_k}
\]
for all $a\in S$.

\subsection{Presentations of semigroups}

Let $\Sigma = \{x_1,\dots,x_d\}$ be a finite alphabet. The free semigroup $\Sigma^*$ consists of all finite strings over the alphabet $\Sigma$.
Every finitely generated semigroup can be described as a quotient of the free semigroup on some finite alphabet, modulo a congruence\footnote{In universal algebra, a congruence on an algebraic structure $A$ is a substructure of $A\times A$ that is also an equivalence relation.}, which, in practice, will be given by a set of relations of the form $u=w$ where $u,w\in \Sigma^*$. For example, the semigroup $S=\left<x_1,x_2|x_1^2=x_1,x_2^2=x_2,x_1x_2x_1=x_2x_1x_2\right>$ is given as a quotient of the free semigroup generated by $x_1,x_2$ modulo the congruence generated by the relations $x_1^2=x_1,x_2^2=x_2,x_1x_2x_1=x_2x_1x_2$; that is, the largest semigroup generated by two elements satisfying these relations.

Endow the alphabet $\Sigma$ with some total order, say, $x_1<\cdots<x_d$. This induces a (lexicographic) total order on the free semigroup $\Sigma^*$ as follows: $u<w$ if either $|u|<|w|$ or if $u=u_1\cdots u_n,\ w=w_1\cdots w_n$ where $u_1,\dots,u_,w_1,\dots,w_n\in \Sigma$ and for the first $i$ such that $u_i\neq w_i$, we have $u_i<w_i$ with respect to the order on $\Sigma$.

Given a system of relations $u_1=w_1,u_2=w_2,\dots$ in the free semigroup $\Sigma^*$ with $u_1<w_1,u_2<w_2,\dots$, consider the \textbf{reduction procedure} $w_1\rightarrow u_1,w_2\rightarrow u_2,\dots$, which can be thought of as an algorithm replacing every element in $\Sigma^*$ containing some $w_i$ as a subword, say, $v=v_0w_iv_1$, by the same word with $w_i$ replaced by $u_i$, namely, $v_0u_iv_1$. This algorithm is not deterministic, as we can have multiple occurrences of one or more $w_i$'s in $v$; some of these occurrences can even overlap. This situation is called an \textbf{ambiguity}. We say that an ambiguity is \textbf{resolvable} if the two different reductions can be further reduced to yield the same irreducible word -- that is, a word that cannot be further reduced by any of the reduction rules.

Bergman's `Diamond Lemma' \cite{Bergman} asserts that, if all ambiguities are resolvable, then the elements of the associated quotient semigroup are in one to one bijection with all irreducible words in the free semigroup. It is worth noting that Bergman's Diamond Lemma in fact applies in a much more general setting, of reduction rules associated with defining relations of associative algebras; however, we focus here on the semigroup-theoretic version for conciseness.

For example, the defining relations of the semigroup $S$ above give the reduction rules $r_1\colon x_1^2\rightarrow x_1,r_2\colon x_2^2\rightarrow x_2,r_3\colon x_2x_1x_2\rightarrow x_1x_2x_1$ (indeed, putting $x_1<x_2$ we have that $x_1x_2x_1 < x_2x_1x_2$). There is an ambiguity in the word $x_1^3$, since the word $x_1^2$, which can be reduced by $r_1$, appears in it twice; however, both applications of $r_1$ yield $x_1^2$, which can be further reduced to the irreducible word $x_1$. A less trivial ambiguity is in the word $x_2^2x_1x_2$, which can be reduced to both $x_2x_1x_2$ (by $r_2$) and to $x_2x_1x_2x_1$ (by $r_3$). However, $x_2x_1x_2$ can be further reduced to the irreducible $x_1x_2x_1$ by $r_3$, and $x_2x_1x_2x_1$ can be further reduced to $x_1x_2x_1^2$ by $r_3$ and then to $x_1x_2x_1$ by $r_1$. This shows that this ambiguity is resolvable; see below.

\begin{center}
\begin{tikzcd}
                              & \substack{(x_2x_2)x_1x_2=\\x_2(x_2x_1x_2)} \arrow[ld, "r_2"'] \arrow[rd, "r_3"] &                                \\
x_2x_1x_2 \arrow[rdd, "r_3"'] &                                                                                 & x_2x_1x_2x_1 \arrow[d, "r_3"]  \\
                              &                                                                                 & x_1x_2x_1x_1 \arrow[ld, "r_1"] \\
                              & x_1x_2x_1                                                                       &
\end{tikzcd}
\end{center}
Using Bergman's Diamond Lemma, it can be shown that $S=\{x_1,x_2,x_1x_2,x_2x_1,x_1x_2x_1\}$.

\section{Construction}\label{sec3}

\begin{cons}\label{cons:Sm}
    Fix an increasing sequence $\vec{m}=\set{m_i}_{i=1}^{\infty}$ of positive integers such that $m_1=1$. For any two positive integers $j,j'\ge 1$, write $j'\prec j$ if there is some $k\ge 1$ such that $j'<m_k\le j$, and write $j'\preceq j$ if $j'\prec j$ or there is some $k\ge 1$ such that $m_k\le j,j'<m_{k+1}$. Define the semigroup
    \[
        \Sm = \sg{x,y \mid x^2=x, \forall i,j\ge 1: xy^jxy^{j'}x=xy^jx \text{ if } j'\prec j}.
    \]
\end{cons}

We think of the defining relations of $\Sm$ as follows. We partition the set of positive integers into equivalence classes $[m_1,m_2),[m_2,m_3),\dots$.
Now $x$ is an idempotent, and $xy^jx$ `absorbs' $xy^{j'}x$ if $j' \prec j$.

\begin{prop}\label{prop:reduce}
    Every element of $a\in\Sm$ can be uniquely written as
    \begin{equation}\label{eq:reduced-form}
        a = y^{j_0}xy^{j_1}x\cdots y^{j_{t-1}}xy^{j_t},
    \end{equation}
    where $j_0,j_t\ge 0$, $j_1,\dots,j_{t-1}\ge 1$ and $j_1\preceq j_2\preceq\cdots\preceq j_{t-1}$.
\end{prop}

We call \eqref{eq:reduced-form} the \textbf{reduced form} of $a$.

\begin{proof}
    We apply Bergman's diamond lemma to the defining presentation of $\mathcal{S}_{\vec{m}}$ with respect to the ordered alphabet $\{x,y\},\ x<y$, along with the reduction rules:
    \begin{eqnarray*}
       r \colon&&  x^2  \longrightarrow   x, \\
      s_{j,j'} \colon &&  x y^j x y^{j'} x  \longrightarrow  x y^j x  \end{eqnarray*}
      whenever $j' \prec j$.
Let us check that all ambiguities arising from overlapping reductions are resolvable.

\begin{center}
\begin{tikzcd}
                    & (xx)x=x(xx) \arrow[ld, "r"'] \arrow[rd, "r"] &                    \\
xx \arrow[rd, "r"'] &                                              & xx \arrow[ld, "r"] \\
                    & x                                            &
\end{tikzcd}
\end{center}

\begin{center}
\begin{tikzcd}
                                       & \substack{(xx)y^jxy^{j'}x \\ =x(xy^jxy^{j'}x)} \arrow[ld, "r"'] \arrow[rd, "{s_{j,j'}}"] &                        \\
xy^jxy^{j'}x \arrow[rd, "{s_{j,j'}}"'] &                                                                                          & xxy^jx \arrow[ld, "r"] \\
                                       & xy^j x                                                                                   &
\end{tikzcd}
\end{center}

\begin{center}
\begin{tikzcd}
                                         & \substack{(xy^jxy^{j'}x)y^{j''}x \\ =xy^j(xy^{j'}xy^{j''}x)} \arrow[ld, "{s_{j,j'}}"'] \arrow[rd, "{s_{j',j''}}"] &                                       \\
xy^jxy^{j''}x \arrow[rd, "{s_{j,j''}}"'] &                                                                                                                   & xy^jxy^{j'}x \arrow[ld, "{s_{j,j'}}"] \\
                                         & xy^j x                                                                                                            &
\end{tikzcd}
\end{center}

\end{proof}

\begin{figure}[h]    \centering    \includegraphics[width=0.5\linewidth]{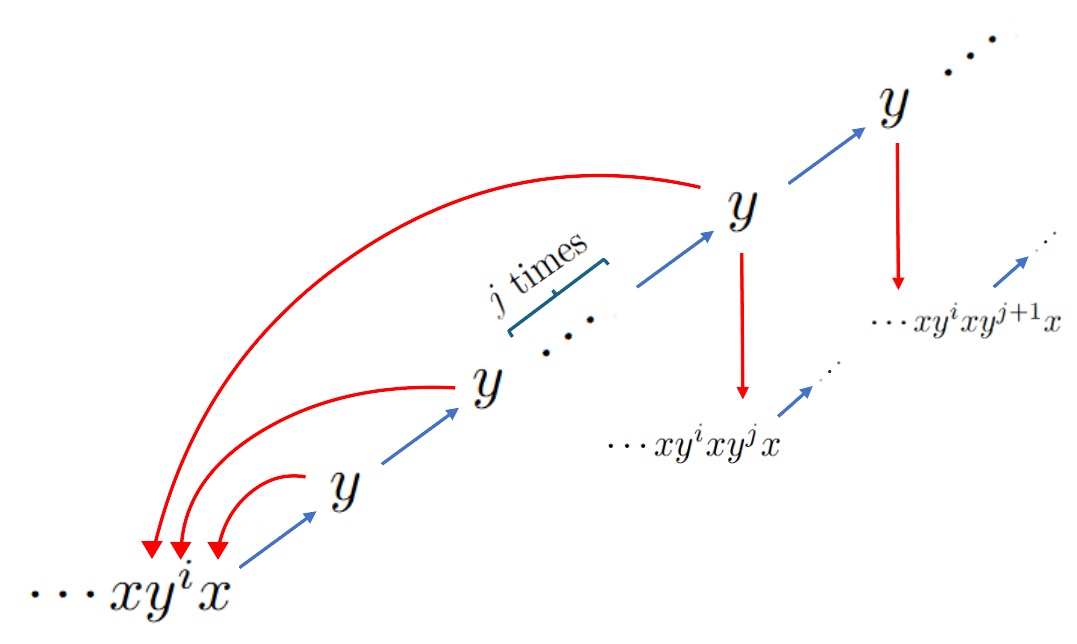}    \caption{A finite part of the Cayley graph of $\Sm$ (here $j\succeq i$). Red arrows correspond to right multiplication by $x$ and blue arrows correspond to right multiplication by $y$. Here $j\succeq i$.}    \label{fig:Cayley_Sm}\end{figure}

\begin{rem}\label{rem:diamond-alg}
    Bergman's diamond lemma also details how to find the reduced form of a given element $a$. First write $a$ as a product of generators. Begin by replacing any occurrence of $x^2$ by $x$. Then replace any occurrence of $xy^jxy^{j'}x$ by $xy^jx$ if $j'\prec j$. Once no more reductions of these forms can be made, the result is the reduced form of $a$.
\end{rem}

\begin{cor}
    In the notations of \Pref{prop:reduce},
    \[
        \left|a\right| = t + \sum_{i=0}^t j_t.
    \]
\end{cor}

\section{Speed estimates}\label{sec4}

Let $R_n=X_1\cdots X_n$ be a simple random walk on~$\Sm$ with respect to the generating set $\set{x,y}$. In this section, we provide estimates for the speed $\E\left|R_n\right|$ of $R_n$. For convenience, we write the reduced form of $R_n$ as
\begin{equation}\label{eq:reduced-form-Rn}
    R_n = y^{N_0(R_n)}\left(\prod_{k=1}^{\infty}\pi_k(R_n)\right)x^{\delta(R_n)}y^{N_{\infty}(R_n)},
\end{equation}
where $\pi_k(R_n)=xy^{j_{k,1}}\cdots xy^{j_{k,N_k(R_n)}}$ is an element such that
\[
    m_k\le j_{k,1},\dots,j_{k,N_k(R_n)}<m_{k+1},
\]
or $\pi_k(R_n)=1$ if no subword of the form $xy^jx$ for $m_k\le j<m_{k+1}$ appears in the reduced form of $R_n$. Also, $\delta(R_n)$ is the indicator of whether there is some $k\in\N$ such that $N_k(R_n)>0$. Under these notations,
\begin{equation}\label{eq:length-Rn}
    \left|R_n\right| = \sum_{k=1}^{\infty}\left|\pi_k(R_n)\right| + N_0(R_n) + N_{\infty}(R_n) + \delta(R_n).
\end{equation}

\begin{rem}\label{rem:R-infty}
    It will be useful to extend $R_n$ to an infinite product, and consider $R_{\infty}=X_1X_2\cdots$ as an infinite word in $x,y$. We note that the reduced form of $R_{\infty}$ is almost surely a well-defined infinite word. Indeed, we may write $R_{\infty}=Y_1Y_2\cdots$, where $Y_1,Y_2,\dots$ are i.i.d.\ random variables with distribution $\Pr(Y_1=y^jx)=\frac{1}{2^{j+1}}$ for all $j=0,1,\dots$. The reduced form of $R_{\infty}$ is then given by replacing each occurrence of $x^2$ by $x$, and removing $Y_t$ from the product if there appears $Y_i$ which absorbs it. Since almost surely the length of $Y_1,Y_2,\dots$ will not be bounded, the prefixes of the infinite word $Y_1Y_2\cdots$ after applying these reductions stabilize, and thus the reduced form of $R_{\infty}$ is well-defined.

    We further note that, if the reduced form of $R_{\infty}$ exists, then $\pi_k(R_n)$ is a subword of $\pi_k(R_{\infty})$ for all $k\in\N\cup\set{0}$, since they are achieved by applying the same reduction steps.
\end{rem}

To establish an upper bound for the speed of $R_n$, we begin with estimating its prefix and suffix of $y$'s, i.e., $N_0(R_n)$ and $N_{\infty}(R_n)$.

\begin{prop}\label{prop:N-0-infty}
    For $k\in\set{0,\infty}$ we have $\E[N_k(R_n)]\le 1$.
\end{prop}

\begin{proof}
    Indeed, $N_0(R_n)$ counts the number of $y$'s that appear in $R_n$ before the first time an $x$ appears in $R_n$. This is bounded by a geometric random variable with parameter $\frac{1}{2}$ supported on $\set{0,1,\dots}$, and thus  $\E[N_0(R_n)]\le 1$. The calculation for $N_{\infty}(R_n)$ is the same in reverse.
\end{proof}

The next two propositions consider $\E\left|\pi_k(R_n)\right|$ for $k\in\N$. \Pref{prop:up-bound-pi-k} is tight for small values of $k$, when we expect $\pi_k(R_n)$ to appear entirely as a subword of $\pi_k(R_{\infty})$ (and, likely, both will be equal). In this case we estimate $\E\left|\pi_k(R_n)\right|$ by estimating $\E\left|\pi_k(R_{\infty})\right|$. On the contrary, \Pref{prop:up-bound-pi-k-inds} is suitable for large values of $k$, when there is a high probability that no word of the form $xy^jx$ for $m_k\le j<m_{k+1}$ appears in $R_n$.

\begin{prop}\label{prop:up-bound-pi-k}
    For any $k\in\N$ we have $\E\left|\pi_k(R_{\infty})\right| \le (m_k+2)2^{m_{k+1}-m_k}$.
\end{prop}

\begin{proof}
    Write $R_{\infty}=Y_1Y_2\cdots$ as in \Rref{rem:R-infty}, and define the random variable $\tau=\inf\set{i\ge 1\mid \left|Y_i\right|\ge m_{k+1}+1}$. Then
    \[
        \left|\pi_k(R_{\infty})\right| = \sum_{i=1}^{\tau-1}\left|Y_i\right|\ind{\set{\left|Y_i\right|\ge m_k+1}}.
    \]

    Since $Y_1,Y_2,\dots$ are independent and $\Pr(\left|Y_i\right|\ge m_{k+1}+1)=2^{-m_{k+1}}$ for all $i$, the random variable $\tau$ is a geometric random variable with parameter $2^{-m_{k+1}}$. Furthermore, note that $\set{\tau=t}=\bigcap_{i=1}^{t-1}\set{\left|Y_i\right|\le m_{k+1}}\cap\set{\left|Y_t\right|\ge m_{k+1}+1}$, so the random variables $Y_1,\dots,Y_m$ are independent and identically distributed also after conditioning by $\set{\tau=m}$. Therefore,
    \begin{align*}
        \E\left|\pi_k(R_{\infty})\right| &= \E\left[\sum_{i=1}^{\tau-1}\left|Y_i\right|\ind{\set{\left|Y_i\right|\ge m_k+1}}\right] \\
        & = \sum_{t=1}^{\infty}\E\left[\sum_{i=1}^{t-1}\left|Y_i\right|\ind{\set{\left|Y_i\right|\ge m_k+1}}\mid \tau=t\right]\Pr(\tau=t) \\
        & = \sum_{t=1}^{\infty} (t-1)\E\left[\left|Y_1\right|\ind{\set{\left|Y_1\right|\ge m_k+1}}\mid \left|Y_1\right|\le m_{k+1}\right]\Pr(\tau=t) \\
        & = \E\left[\left|Y_1\right|\ind{\set{\left|Y_1\right|\ge m_k+1}}\mid \left|Y_1\right|\le m_{k+1}\right] (\E[\tau]-1) \\
        & = \E\left[\left|Y_1\right|\mid m_k+1\le \left|Y_1\right|\le m_{k+1}\right]\Pr\left(\left|Y_1\right|\ge m_k+1\right)(\E[\tau]-1) \\
        & = \E\left[\left|Y_1\right|\mid m_k+1\le \left|Y_1\right|\le m_{k+1}\right]\Pr\left(\left|Y_1\right|\ge m_k+1\right)(2^{m_{k+1}}-1).
    \end{align*}
    Next, note that $\left|Y_1\right|$ has a geometric distribution with parameter $\frac{1}{2}$, and thus $\Pr\left(\left|Y_1\right|\ge m_k+1\right)=2^{-m_k}$ and
    \begin{align*}
        \E\left[\left|Y_1\right|\mid m_k+1\le \left|Y_1\right|\le m_{k+1}\right] &= \sum_{i=m_k+1}^{m_{k+1}}i\,\Pr(\left|Y_1\right|=i\mid m_k+1\le \left|Y_1\right|\le m_{k+1}) \\
        & = \frac{1}{2^{-m_k}-2^{-m_{k+1}}}\sum_{i=m_k+1}^{m_{k+1}}i 2^{-i-1} \\
        & \le 2^{m_k+1}\sum_{i=m_k+1}^{\infty}i 2^{-i-1} = m_k+2.
    \end{align*}
    Combining everything, we have
    \[
        \E\left|\pi_k(R_{\infty})\right| \le (m_k+2)2^{m_{k+1}-m_k},
    \]
    as required.
\end{proof}

\begin{prop}\label{prop:up-bound-pi-k-inds}
    For any $k\in\N$ we have $\E\left|\pi_k(R_n)\right| \le \frac{m_k+2}{2^{m_k+1}}n$.
\end{prop}

\begin{proof}
    Consider again $R_{\infty}=X_1X_2\cdots$ as an infinite word in the generators $x,y$, and $R_n$ as its prefix consisting of the first $n$ letters. For each $1\le t\le n$, let $W_t$ the number of $y$'s that appear from position $t+1$ until the first $x$. Then
    \[
        \left|\pi_k(R_n)\right| \le \sum_{t=1}^n \ind{\set{X_t=x}}\ind{\set{m_k\le W_t<m_{k+1}}}(W_t+1).
    \]
    Indeed, the right hand side sums the lengths of all appearances of $xy^jx$ for $m_k\le j<m_{k+1}$ in $R_n$ before applying any reductions. Therefore
    \[
        \E\left|\pi_k(R_n)\right| \le \sum_{t=1}^n \E[W_t+1\mid X_t=x,m_k\le W_t<m_{k+1}]\Pr(X_t=x,m_k\le W_t<m_{k+1}).
    \]
    By the Markov property of random walks, $W_t$ is independent of $X_t$, and thus
    \begin{align*}
        \E\left|\pi_k(R_n)\right| &\le \sum_{t=1}^n \E[W_t+1\mid m_k\le W_t<m_{k+1}]\Pr(X_t=x)\Pr(m_k\le W_t<m_{k+1}) \\
        & \le \frac{1}{2^{m_k+1}}\sum_{t=1}^n \E[W_t+1\mid m_k\le W_t<m_{k+1}].
    \end{align*}

    Note that $W_t+1$ is again a geometric random variable with parameter $\frac{1}{2}$. In a similar manner to the proof of \Pref{prop:up-bound-pi-k}, $\E[W_t+1\mid m_k\le W_t<m_{k+1}]\le m_k+1$ for all $t$. Therefore we conclude that
    \[
        \E\left|\pi_k(R_n)\right| \le \frac{m_k+2}{2^{m_k+1}}n.
    \]
\end{proof}

Together, we reach the following upper bound:

\begin{cor}
    \[
        \E\left|R_n\right| \le \sum_{k:m_k\le n-2}\frac{m_k+2}{2^{m_k}}\min\set{2^{m_{k+1}},n}+3.
    \]
\end{cor}

\begin{proof}
    By \eqref{eq:length-Rn},
    \begin{align*}
        \E\left|R_n\right| &= \sum_{k=1}^{\infty}\E\left|\pi_k(R_n)\right| + \E[N_0(R_n)] + \E[N_{\infty}(R_n)] + \E[\delta(R_n)] \\
        & \le \E\left|R_n\right| = \sum_{k=1}^{\infty}\E\left|\pi_k(R_n)\right| + 3.
    \end{align*}
    Note also that $N_k(R_n)=0$ for any $k$ with $m_k>n-2$, since then $xy^jx$ is a word of length $j+2>n$ for any $m_k\le j<m_{k+1}$, and thus cannot be a subword of $R_n$. It is therefore enough to prove that $\E\left|\pi_k(R_n)\right|\le \frac{m_k+2}{2^{m_k}}\min\set{2^{m_{k+1}},n}$ for all $k\in\N$ with $m_k\le n-2$. Indeed, note first that $\pi_k(R_n)$ is a subword of $\pi_k(R_{\infty})$ and both are reduced, hence by \Pref{prop:up-bound-pi-k} we have
    \[
        \E\left|\pi_k(R_n)\right| \le \E\left|\pi_k(R_{\infty})\right| \le (m_k+2)2^{m_{k+1}-m_k}.
    \]
    The fact that $\E\left|\pi_k(R_n)\right|\le\frac{m_k+2}{2^{m_k+1}}n$ follows from \Pref{prop:up-bound-pi-k-inds}.
\end{proof}

For the lower bound on the speed of $R_n$, we focus on bounding $\E\left|\pi_k(R_n)\right|$ for $k\in\N$ from below. We do this by estimating the number of subwords $xy^jx$ for $m_k\le j<m_{k+1}$ of an appropriate prefix of $R_n$.

\begin{prop}\label{prop:low-bound-N-k}
    For any $k\in\N$ such that $m_k\le n-2$, we have $$ \E[N_k(R_n)]\ge \frac{\min\set{2^{m_{k+1}},n}}{2^{m_k+3}}. $$
\end{prop}

\begin{proof}
    Write $R_n=X_1\cdots X_n$. For any $1\le i\le 2^{m_{k+1}}$, let $A_i$ denote the event that there is some $m_k\le j<m_{k+1}$ such that $X_i\cdots X_{i+j+1}=xy^jx$, and that $y^{m_{k+1}}$ is not a subword of $X_1\cdots X_{i-1}$. We first claim that $N_k(R_n)\ge\sum_{i=1}^{\min\set{2^{m_{k+1}-1},n}}\ind{A_i}$. Indeed, $N_k(R_n)$ counts the appearances of subwords of the form $xy^jx$ for $m_k\le j<m_{k+1}$ in the reduced form of $R_n$, and the event $A_i$ precisely means that such a subword starts at position $i$. Therefore
    \[
        \E[N_k(R_n)] \ge \sum_{i=1}^{\min\set{2^{m_{k+1}-1},n}} \E[\ind{A_i}] = \sum_{i=1}^{\min\set{2^{m_{k+1}-1},n}}\Pr(A_i).
    \]

    Next, we write $A_i=B_i\cap C_i$, where $B_i$ is the event that $xy^{m_{k+1}}$ is not a subword of $X_1\cdots X_{i-1}$, and $C_i$ is the event that $X_i\cdots X_{i+j+1}=xy^jx$ for some $m_k\le j<m_{k+1}$. Conditioned on $X_i=x$, the events $B_i$ and $C_i$ are independent. Indeed, conditioned on $X_i=x$, the event $B_i$ depends only on $X_1,\cdots,X_{i-1}$, and $C_i$ depends only on $X_{i+1},\cdots,X_n$. Therefore, for any $1\le i\le n-m_k-1$ we have
    \begin{eqnarray*}
        \Pr(A_i) = \Pr(B_i\cap C_i) & = & \frac{1}{2}\Pr(B_i\cap C_i\mid X_i=x) \\ & = & \frac{1}{2}\Pr(B_i\mid X_i=x)\Pr(C_i\mid X_i=x).
    \end{eqnarray*}
    We claim that $\Pr(B_i\mid X_i=x)\ge\frac{1}{2}$. Indeed, let $M$ denote the number of appearances of $y^{m_{k+1}}$ in $X_1\cdots X_{i-1}$. At any given position, the probability that $xy^{m_{k+1}}$ starts at that position is $2^{-m_{k+1}-1}$, and thus $\E[M]\le(i-1)2^{-m_{k+1}-1}\le\frac{1}{2}$. By Markov's inequality, $\Pr(M\ge 1\mid X_i=x)\le\frac{1}{2}$, and thus $\Pr(B_i\mid X_i=x)\ge\frac{1}{2}$. Finally, note that $\Pr(C_i\mid X_i=x)\ge\Pr(X_{i+1}\cdots X_{i+m_k}=y^{m_k})\ge 2^{-m_k}$.

    Combining all inequalities, we have
    \[
        \Pr(A_i) \ge 2^{-m_k-2}\,,
    \]
    and thus
    \[
        \E[N_k(R_n)] \ge \sum_{i=1}^{\min\set{2^{m_{k+1}-1},n}}\Pr(A_i) \ge \frac{\min\set{2^{m_{k+1}},n}}{2^{m_k+3}},
    \]
    completing the proof.
\end{proof}

Combining all values of $k$, we get:

\begin{cor}
    \[
        \E\left|R_n\right| \ge \sum_{k:m_k\le n-2}\frac{m_k+1}{2^{m_k+3}}\min\set{2^{m_{k+1}},n}.
    \]
\end{cor}

\begin{proof}
    By \eqref{eq:length-Rn}, it follows that
    \begin{eqnarray*}
        \E\left|R_n\right| & \ge & \sum_{k=1}^{\infty}\E\left|\pi_k(R_n)\right| \\ &  \ge & \sum_{k=1}^{\infty}(m_k+1)\E[N_k(R_n)] \\ & \ge &  \sum_{k:m_k\le n-2}(m_k+1)\E[N_k(R_n)],
    \end{eqnarray*}
    and the claim now follows from \Pref{prop:low-bound-N-k}.
\end{proof}

Our upper and lower bounds on the speed of $R_n$ are summarized in this corollary:

\begin{cor} \label{cor speed gamma}
  We have
  \[
    \sum_{k:m_k\le n-2}\frac{m_k+1}{2^{m_k+3}}\min\set{2^{m_{k+1}},n} \le \E\left|R_n\right| \le \sum_{k:m_k\le n-2}\frac{m_k+2}{2^{m_k}}\min\set{2^{m_{k+1}},n}+3.
  \]
\end{cor}

\section{Realization of speed functions}\label{sec5}

Given a sequence $\vec{m}=(m_1,m_2,\dots)$ (with $m_1=1$) 
,  denote:
\[
\gamma_{\vec{m}}(n) = \sum_{k:\ m_{k+1} \leq \log_2 n} m_k 2^{m_{k+1}-m_k}.
\]

\begin{lem} \label{lem:speed gamma} Suppose that, for some constants $\beta \geq 1$ and $\delta>0$ we have that $m_i\leq \beta m_{i-1} + \delta$ for all $i > 1$. Then
\[
C_1\gamma_{\vec{m}}(n) \leq \E|R_n| \leq C_2 \left( \gamma_{\vec{m}}(n) + n^{\frac{\beta-1}{\beta}}\log_2 n \right)
\]
for some positive constants $C_1,C_2>0$.
\end{lem}
\begin{proof}
By Corollary \ref{cor speed gamma},

\[
\E|R_n| \geq \sum_{k:\, m_{k+1}\le\log_2 n} (m_k+1) 2^{m_{k+1}-m_k-3} \geq \frac{1}{8}\gamma_{\vec{m}}(n).
\]

On the other hand,
\begin{eqnarray*}
    && \sum_{k=1}^{\infty} \frac{m_k + 2}{2^{m_k}} \min\left\{2^{m_{k+1}},n \right\} \leq \nonumber \\ && \leq \sum_{k:\ m_{k+1} \leq \log_2 n} (m_k + 2) 2^{m_{k+1}-m_k} +
 \sum_{k:\ m_{k+1}>\log_2 n} \frac{(m_k + 2)n}{2^{m_k}}  \nonumber \\ && \leq 3 \gamma_{\vec{m}}(n) + 3n \sum_{k:\ m_{k+1}>\log_2 n}\frac{m_k}{2^{m_k}}
\end{eqnarray*}
and
\begin{eqnarray*}
\sum_{k:\ m_{k+1}>\log_2 n}\frac{m_k}{2^{m_k}} \leq \sum_{\substack{r\in \mathbb{Z}: \\ r\geq \frac{1}{\beta} \log_2 n - \frac{\delta}{\beta}}}\frac{r}{2^{r}}.
\end{eqnarray*}
Fix $N>1$. Let $g(t)=\sum_{r = N}^{\infty} rt^r=t\left(\frac{t^N}{1-t}\right)'=\frac{t^N}{(1-t)^2}(N(1-t)+t)$. Therefore $g(1/2)\leq N2^{-N-2}$, so, taking $N=\lceil \frac{1}{\beta}\log_2 n - \frac{\delta}{\beta} \rceil$,
\begin{eqnarray*}
\sum_{k:\ m_{k+1}>\log_2 n}\frac{m_k}{2^{m_k}} \leq \left( \frac{1}{\beta}\log_2 n -\frac{\delta}{\beta} \right)2^{-\frac{1}{\beta}\log_2n + \frac{\delta}{\beta}} \leq C n^{-\frac{1}{\beta}}\log_2 n
\end{eqnarray*}
for $C=2^{\frac{\delta}{\beta}}\cdot \frac{1}{\beta}$.
Hence:
\begin{eqnarray*}
\sum_{k=1}^{\infty} \frac{m_k + 2}{2^{m_k}} \min\left\{2^{m_{k+1}},n \right\} \leq 3 \gamma_{\vec{m}}(n) + 3n \cdot Cn^{-\frac{1}{\beta}}\log_2 n.
\end{eqnarray*}
Collecting pieces, using Corollary \ref{cor speed gamma},
\begin{eqnarray*}
\frac{1}{8}\gamma_{\vec{m}}(n) \leq \E|R_n| \leq 3\gamma_{\vec{m}}(n)+ 3 C n^{1-\frac{1}{\beta}}\log_2 n+3,
\end{eqnarray*}
and the claim follow.
\end{proof}

This lemma allows us to understand the slowest rate of speed functions our construction yields:

\begin{exmpl}
    Take the sequence $m_k=k$. Then
    \[
        \gamma_{\vec{m}}(n) = \sum_{k=1}^{\floor{\log_2 n}-1} 2k = \binom{\floor{\log_2 n}}{2}.
    \]
    We may therefore apply \Lref{lem:speed gamma} with $\beta=1$ and $\delta=1$, and deduce that the speed of the random walk $R_n$ on the semigroup $\Sm$ (corresponding to the sequence $m_k=k$) satisfies $\E\left|R_n\right|\asymp\log^2 n$.
\end{exmpl}

\subsection{Arbitrary speed exponents}

We can now prove:
\begin{thm}[{Theorem A}]
Let $\alpha\in (0,1)$. Then there exists a finitely generated semigroup such that, for a simple random walk supported on a finite generating set of it, $$\limsup_{n\rightarrow \infty} \log_n \E|R_n| = \alpha.$$
\end{thm}

\begin{proof}
Fix a real number $\alpha\in (0,1)$ and let  $\beta>1$ be such that $\frac{\beta-1}{\beta}=\alpha$ (namely, $\beta = \frac{1}{1-\alpha}$). For each $i\geq 1$, let $m_{i}= 1+\lfloor \beta \rfloor+\cdots+\lfloor \beta^{i-1} \rfloor$.

Let $N\in \mathbb{N}$ and let $t$ be such that $\log_2 N\in [m_t,m_{t+1})$. Then
\[
\gamma_{\vec{m}}(N) = \sum_{k:\ m_{k+1} \leq \log_2 N} m_k 2^{m_{k+1}-m_k} = \sum_{k=1}^{t-1} m_k 2^{m_{k+1}-m_k} = \gamma_{\vec{m}} (2^{m_t}).
\]
Let us focus on $n=2^{m_t}$ for $t\gg_\beta 0$. Then
\begin{eqnarray*} \frac{\beta}{\beta - 1}\cdot \beta^{t-1} - 2t \leq \frac{\beta^t - 1}{\beta - 1} - t & \leq & 1+\lfloor \beta \rfloor+\cdots+\lfloor \beta^{t-1} \rfloor \\ & \leq & \frac{\beta^t - 1}{\beta - 1} \leq \frac{\beta}{\beta - 1}\cdot \beta^{t-1}. \end{eqnarray*}
Notice that $$\log_2\log_2 n\geq \log_2 \lfloor \beta^{t-1} \rfloor \geq (t-1)\log_2 \beta - 1 \geq ct$$ for some $c>0$ depending only on $\beta$ (recall that $t\gg_\beta 0$). Hence $$\log_2 n = \frac{\beta}{\beta - 1} \cdot \beta^{t-1} + \Delta,\ \ \ \text{where}\ \ \  \Delta\leq 2t\leq \frac{2}{c}\log_2 \log_2 n.$$
It follows that
\begin{eqnarray*}
\gamma_{\vec{m}}(n) & \geq & m_{t-1} 2^{m_t - m_{t-1}}\geq 2^{\lfloor \beta^{t-1} \rfloor} \\ & \geq & \frac{1}{2} 2^{\beta^{t-1}}\geq \frac{1}{2} 2^{\frac{\beta-1}{\beta}\log_2 n - \frac{2}{c}\log_2\log_2 n} \geq n^{\frac{\beta-1}{\beta}}\cdot \log_2^{-\frac{3}{c}} n
\end{eqnarray*} (recall that $t\gg_\beta 0$).
On the other hand,
\begin{eqnarray*}
\gamma_{\vec{m}}(N) & = & \gamma_{\vec{m}}(2^{m_t}) = \sum_{k=1}^{t-1} m_k 2^{m_{k+1}-m_k} \\ & \leq & m_{t-1}(2^{m_2-m_1}+\cdots+2^{m_t-m_{t-1}}) \\ & \leq & 2m_{t-1} 2^{m_t - m_{t-1}} \leq 2\log_2 N \cdot  2^{ \beta^{t-1} } \\ & \leq & 2\log_2 N \cdot 2^{\frac{\beta - 1}{\beta}m_t+\frac{2(\beta-1)}{\beta}t} \\ & \leq & 2 \log_2^{1+\frac{2(\beta-1)}{c\beta}} N \cdot N^{\frac{\beta-1}{\beta}}.
\end{eqnarray*}
It follows that
\[
\limsup_{n\rightarrow \infty} \log_n\gamma_{\vec{m}}(n) = \frac{\beta-1}{\beta} = \alpha.
\]
Finally, by Lemma \ref{lem:speed gamma},
\begin{eqnarray*}
\alpha = \limsup_{n\rightarrow \infty} \log_n\gamma_{\vec{m}}(n) & \leq & \limsup_{n\rightarrow \infty} \log_n \E|R_n| \\ & \leq & \max\Big\{\limsup_{n\rightarrow \infty} \log_n\gamma_{\vec{m}}(n),\frac{\beta-1}{\beta}\Big\} = \alpha.
\end{eqnarray*}
The proof is completed.
\end{proof}

\section{Arbitrarily slow speed functions}\label{sec6}

In this section, we prove Theorem B. Namely, we show that there is no gap between constant and non-constant speed functions of random walks on semigroups. To do this, we analyze the speed of random walks on a certain quotient of $\Sm$.

The quotient we consider is the following:

\begin{cons}
    Fix an increasing sequence $\vec{m}=\set{m_i}_{i=1}^{\infty}$ of positive integers such that $m_1=1$. We use the notations $\prec$ and $\preceq$ as in \Csref{cons:Sm}. We also write $j\sim j'$ if $j\preceq j'\preceq j$, i.e., if $m_k\le j,j'<m_{k+1}$ for some $k\in\N$. We define the semigroup
    \[
        \SSm = \sg{x,y \mid x^2=x, \forall i,j\ge 1: xy^jxy^{j'}x=xy^jx \text{ if } j'\preceq j}.
    \]
\end{cons}

The difference between $\Sm$ and $\SSm$ comes when considering the product of $xy^jx$ with $xy^{j'}x$ when $j\sim j'$. Indeed, in $\Sm$, the subsemigroup $\sg{xy^{m_k}x,\dots,xy^{m_{k+1}-1}x}$ is free for all $k$, whereas the same semigroup in $\SSm$ satisfies the semigroup law $ab=a$.

We provide the following upper bound for the speed of a simple random walk on $\SSm$:

\begin{prop}\label{prop:speed-SSm}
    Let $R_n$ be a simple random walk on $\SSm$ with respect to the generating set $\set{x,y}$. Then
    \[
        c_0\sum_{k:m_k\le\log_2 n}(m_k+1)\le\E\left|R_n\right| \le \sum_{k:m_k\le n-2} (m_k+2)\min\set{1,\frac{n}{2^{m_k}}}+3.
    \]
\end{prop}

The proof of the proposition follows similar lines as the calculations in \Sref{sec4}. We first note that, as in \Pref{prop:reduce}, every element $a\in\SSm$ has a reduced form
\[
    a=y^{j_0}xy^{j_1}x\cdots y^{j_{t-1}}xy^{j_t},
\]
where $j_0,j_t\ge 0$, $j_1,\dots,j_{t-1}\ge 0$, and $j_1\prec j_2\prec\cdots\prec j_{t-1}$ (rather than $\preceq$ in $\Sm$). We write the reduced form of the simple random walk $R_n$ as
\begin{equation}\label{eq:reduced-form-Rn-SSm}
    R_n = y^{N_0(R_n)}\left(\prod_{k=1}^{\infty}\overline{\pi}_k(R_n)\right)x^{\delta(R_n)}y^{N_{\infty}(R_n)},
\end{equation}
where $\overline{\pi}_k(R_n)=xy^{j}$ is an element such that $m_k\le j<m_{k+1}$, or $\overline{\pi}_k(R_n)=1$ if no subword of the form $xy^jx$ for $m_k\le j<m_{k+1}$ appears in the reduced form of $R_n$. Also, $\delta(R_n)$ is the indicator of whether there is some $k$ such that $\overline{\pi}_k(R_n)\ne 1$. We have again
\[
    \left|R_n\right| = \sum_{k=1}^{\infty}\left|\overline{\pi}_k(R_n)\right| + N_0(R_n) + N_{\infty}(R_n) + \delta(R_n),
\]
so
\begin{equation}\label{eq:speed-SSm}
    \E\left|R_n\right| = \sum_{k=1}^{\infty}\E\left|\overline{\pi}_k(R_n)\right| + \E[N_0(R_n)] + \E[N_{\infty}(R_n)] + \E[\delta(R_n)].
\end{equation}
\Pref{prop:N-0-infty} is the same for random walks on $\Sm$ and $\SSm$, so we focus on $\E\left|\overline{\pi}_k(R_n)\right|$.

\begin{prop}\label{prop:SSm-pi-k-up}
    For any $k\in\N$ we have $\E\left|\overline{\pi}_k(R_n)\right|\le m_k+2$.
\end{prop}

\begin{proof}
    Similarly to \Pref{prop:up-bound-pi-k}, we may extend $R_n$ to an infinite word $R_{\infty}$, which has a well-defined reduced form, and thus $\E\left|\overline{\pi}_k(R_n)\right|\le\E\left|\overline{\pi}_k(R_{\infty})\right|$. Write $R_{\infty}=Y_1Y_2\cdots$, where $Y_1,Y_2,\dots$ are i.i.d.\ random variables with distribution $\Pr(Y_1=y^jx)=2^{-j-1}$, and let $\tau=\inf\set{j\mid m_k<\left|Y_j\right|\le m_k+1,\forall i<j:\left|Y_i\right|\le m_k}$ (which might be infinite). Then
    \[
        \left|\overline{\pi}_k(R_{\infty})\right| = \sum_{j=1}^{\infty}\left|Y_j\right|\ind{\set{\tau=j}}.
    \]
    Therefore
    \[
        \E\left|\overline{\pi}_k(R_{\infty})\right| = \sum_{j=1}^{\infty}\E\left[\left|Y_j\right|\ind{\set{\tau=j}}\right]=\sum_{j=1}^{\infty}\E\left[\left|Y_j\right|\mid \tau=j\right]\Pr(\tau=j).
    \]
    By the independence of $Y_1,\dots,Y_j$,
    \[
        \E\left|\overline{\pi}_k(R_{\infty})\right| = \sum_{j=1}^{\infty}\E\left[\left|Y_j\right|\mid m_k+1\le\left|Y_j\right|\le m_{k+1}\right]\Pr(\tau=j).
    \]
    We saw in the proof of \Pref{prop:up-bound-pi-k} that $\E\left[\left|Y_j\right|\mid m_k+1\le\left|Y_j\right|\le m_{k+1}\right]\le m_k+2$, and thus
    \[
        \E\left|\overline{\pi}_k(R_{\infty})\right| \le \sum_{j=1}^{\infty}(m_k+2)\Pr(\tau=j)\le m_k+2.
    \]
\end{proof}

\begin{prop}\label{prop:SSm-pi-k-low}
    There exists a universal constant $c_0>0$ such that, for any $k\in\N$ such that $m_k\le \log_2 n$, we have $\Pr(\overline{\pi}_k(R_n)\ne 1)\ge c_0$.
\end{prop}

\begin{proof}
    Consider $R_n=X_1\cdots X_n$ as a word of length $n$. Then $\overline{\pi}_k(R_n)=1$ if and only if no subword of the form $xy^jx$ for $m_k\le j<m_{k+1}$ appears before the first appearance of $xy^{m_{k+1}}$ in $R_n$. Write $m=\min\set{n,2^{m_k}}$. 
    By \Pref{prop:subwords}, the probability that $xy^{m_k}$ appears as a subword in $X_1\cdots X_m$ is bounded from below by some absolute constant $c>0$. After the first appearance of $xy^{m_k}$, the probability that the next letter is $x$ is $\frac{1}{2}$. This shows that the probability that $xy^{m_k}x$ appears before any $xy^j$ for $j\ge m_{k+1}$ is bounded from below by $\frac{1}{2}c$, proving the proposition.
\end{proof}

We can now prove the speed estimates on $R_n$.

\begin{proof}[Proof of \Pref{prop:speed-SSm}]
    Combining \Pref{prop:SSm-pi-k-up} and \Pref{prop:up-bound-pi-k-inds}, which is still valid since $\SSm$ is a quotient of $\Sm$, we have
    \[
        \E\left|\overline{\pi}_k(R_n)\right| \le (m_k+2)\min\set{1,\frac{n}{2^{m_k}}}
    \]
    for all $k\in\N$ such that $m_k\le n-2$. The claimed upper bound now follows from \eqref{eq:speed-SSm}.

    For the lower bound, note that
    \[
        \E\left|\overline{\pi}_k(R_n)\right| \ge (m_k+1)\Pr(\overline{\pi}_k(R_n)\ne 1)\ge c_0(m_k+1)
    \]
    for all $k$ with $m_k\le\log_2 n$, where the last inequality follows from \Pref{prop:SSm-pi-k-low}. Therefore, by \eqref{eq:speed-SSm},
    \[
        \E\left|R_n\right|\ge\sum_{k:m_k\le\log_2 n}\E\left|\overline{\pi}_k(R_n)\right| \ge c_0\sum_{k:m_k\le\log_2 n}(m_k+1)
    \]
    as required.
\end{proof}

We return to the proof of Theorem B. The key observation is that the upper bound of \Pref{prop:speed-SSm} is constant on any interval of the form $[2^{m_k},m_{k+1}+1)$. Therefore, if the sequence $\set{m_k}$ grows very fast, the speed of $R_n$ will grow slowly.

\begin{proof}[Proof of Theorem B]
    Let $\omega\colon\N\to[1,\infty)$ be a function such that $\omega(n)\xrightarrow{n\to\infty}\infty$. We construct a sequence $\vec{m}=\set{m_k}$ inductively as follows: we set $m_1=1$, and for each $k\ge 2$ we let $m_k>2^{m_{k-1}}$ be the first integer such that
    \[
        \omega(m_k) > \sum_{i=1}^{k-1} (m_i+2)+3.
    \]

    Consider the random walk $R_n$ on $\SSm$. By \Pref{prop:speed-SSm}, $\E\left|R_n\right|$ is not bounded, and for all $k\ge 2$ we have
    \[
        \E\left|R_{m_k}\right| \le \sum_{i=1}^{k-1}(m_i+2)\min\set{1,\frac{m_k}{2^{m_i}}}+3 = \sum_{i=1}^{k-1}(m_i+2)+3  < \omega(m_k)
    \]
    proving the theorem.
\end{proof}

\section{A lower bound on the distance of random walks}\label{sec7}

As another angle to study the rate of escape of random walks on semigroups, we ask how close can they be to their starting point. We show that almost surely there is a subsequence of the path of the random walk which ``diverges to infinity'', provided that the semigroup does not contain finite right ideals. The main result of this section is proved for the more general case of random walks on rooted directed graphs, and we rephrase it in the language of random walks on semigroups at the end of the section.

Let $(G,o)$ be a rooted directed graph, and let $R_n$ be a simple random walk on~$G$. If $G$ contains a finite strongly connected component, then there is a positive probability that $R_n$ will reach that component after finitely many steps, and then~$R_n$ will be trapped there from that time onwards. To detect such components, we introduce the following function:

\begin{defn}
    Let $G=(V,E)$ be a rooted directed graph with root $o$. We define
    \[
        \ff_G(v,n) = \max_{w\in B_G(v,n)}\dist(o,w)
    \]
    for each $v\in V$, and the \textbf{rooted ball spread} of $G$ as
    \[
        \ff_G(n) = \min_{v\in V}\ff_G(n,v).
    \]
    In words, $\ff_G(n)\ge r$ if any ball $B_G(v,n)$ of radius $n$ in $G$ contains a vertex $w$ of distance at least $r$ from the root $o$.
\end{defn}

We now describe some properties of the function $\ff_G$.

\begin{rem}
    \begin{enumerate}
        \item For any $n\in\N$ we have $\ff_G(n)\le\ff_G(o,n)\le n$.
        \item For any $v\in V$ and $n\in\N$ we have $\ff_G(v,n)\ge\dist(o,v)$.
    \end{enumerate}
\end{rem}

\begin{rem} \label{rem:undirected}
    If $G=(V,E)$ is undirected, then $\ff_G(n)\ge\frac{n}{2}$ for all $n$. Indeed, we will show that $\ff_G(v,n)\ge\frac{n}{2}$ for all $v\in V$. On the one hand, if $\dist(o,v)\ge\frac{n}{2}$, then $\ff_G(v,n)\ge\dist(o,v)\ge\frac{n}{2}$. On the other hand, if $\dist(o,v)<\frac{n}{2}$, take $v'\in V$ such that $\dist(o,v')=\ceil{\frac{n}{2}}$. Therefore $\dist(v,v')\le n$, and so $\ff_G(v,n)\ge\dist(o,v)\ge\frac{n}{2}$.

    We remark that if $G$ has no dead ends, then $\ff_G(n)=n$, since we can then choose the vertex $v'$ so that $\dist(v,v')=\dist(o,v)+\dist(v,v')$. For this reason, for undirected $G$ we have $\ff_G(n)=n$ for all $n$ if and only $G$ has no dead ends.
\end{rem}

Note that this argument clearly fails for directed graphs, since if a directed graph contains a dead end, there might be no way of going back to the root. However, as mentioned above, we will now show that the function $\ff_G$ detects finite strongly connected components in directed graphs:

\begin{prop}\label{prop:f-bound}
    Let $G=(V,E)$ be a rooted directed graph with root $o$.
    \begin{enumerate}
        \item If $G$ has a finite strongly connected component, then $\ff_G$ is bounded.
        \item If $G$ has no finite strongly connected components, then $\ff_G(\left|B_G(o,n-1)\right|)\ge n$ for all $n\in\N$.
    \end{enumerate}
\end{prop}

\begin{proof}
    \begin{enumerate}
        \item Assume that $G$ has a finite strongly connected component $A$. Fix $v\in A$, and take $R>0$ such that $A\sub B_G(o,R)$. Then
        \[
            \ff_G(n) \le \ff_G(v,n) \le \max_{w\in A}\dist(o,w)\le R
        \]
        for all $n\in\N$, showing that $\ff_G$ is bounded.

        \item Assume that $G$ has no finite strongly connected components, and take $v\in V$. We claim that $\ff_G(v,\left|B_G(o,n-1)\right|)\ge n$ for all $n\in\N$. Indeed, if $\dist(o,v)\ge n$, then
        \[
            \ff_G(v,\left|B_G(o,n-1)\right|)\ge\ff_G(v,n)\ge n
        \]
        since $\left|B_G(o,n-1)\right|\ge n$.

        Suppose therefore that $\dist(o,v)<n$. Since $v$ does not lie in a finite strongly connected component, there is a path $v_0=v,v_1,\dots,v_k$ such that $\dist(o,v_k)=n$. We may assume that the path is simple, i.e.\ $v_i\ne v_j$ for $i\ne j$, and that $k$ is the minimal index such that $\dist(o,v_k)=n$. Therefore $v_0,\dots,v_{k-1}$ are distinct elements in $B_G(o,n-1)$, and thus $k\le\left|B_G(o,n-1)\right|$. It follows that
        \[
            \ff_G(v,\left|B_G(o,n-1)\right|)\ge\ff_G(v,k)\ge \dist(o,v_k)=n,
        \]
        concluding the proof.\qedhere
    \end{enumerate}
\end{proof}

\begin{cor}\label{cor:f-bound}
    Suppose that $G=(V,E)$ has no finite strongly connected components. Suppose further that the $\deg^+(v)\le d$ for all $v\in V$. Then $\ff_G(n)\ge\log_d n-1$ for all $n\in\N$.
\end{cor}

\begin{proof}
    Since $\deg^+(v)\le d$ for all $v\in V$, it follows that $\left|B_G(o,R-1)\right|\le\sum_{i=0}^{R-1} d^i\le d^R$ for all $R\ge 0$. Let $n\in\N$, and take $R$ so that $d^R\le n<d^{R+1}$. Then
    \[
        \ff_G(n) \ge \ff_G(d^R) \ge \ff_G(B_G(o,R-1)) \ge R \ge \log_d n-1
    \]
    as required.
\end{proof}

We now return to studying random walks, and prove Theorem C. I.e., we give a lower bound for the distance of the random walk from the root in terms of the function $\ff_G$.

\begin{thm}\label{thm:low-bound-io}
    Let $G=(V,E)$ be a rooted directed graph without finite strongly connected components, such that $\deg^+(v)\le d$ for all $v\in V$. Let $R_n$ be a simple random walk on $G$. Then, almost surely, $\dist(o,R_t)\ge F(\log_d t+\log_d\log_d t)$ for infinitely many values of $t$.
\end{thm}

\begin{proof}
    We first replace $G$ by the directed graph $G'$ obtained from $G$ by adding loops for any vertex $v\in V$, until $\deg^+_{G'}(v)=d$. We will prove the theorem for the simple random walk $R'_n$ on $G'$, which will also prove the theorem for $R_n$. Indeed, note first that $\ff_G=\ff_{G'}$, since we only added loops. We couple $R_n$ with $R'_n$ as follows: sample $R'_0,R'_1,\dots$ as a simple random walk on $G'$, and then define $R_n$ by taking the sequence $R'_0,R'_1,\dots$ and removing any transition where the random walk $R'_n$ used a loop that we added from $G$ to $G'$. Suppose that for some $t\ge 1$ we have $\dist(o,R'_t)\ge\ff_{G'}(\log_d t+\log_d\log_d t)$; then, by the coupling, there is some $s\le t$ such that $R_s=R'_t$, and thus
    \[
        \dist(o,R_s) = \dist(o,R'_t) \ge \ff_{G'}(\log_d t+\log_d\log_d t) \ge \ff_G(\log_d s+\log_d\log_d s).
    \]

    We therefore focus on the random walk $R'_n$ on $G'$. Color the edges of $G'$ by elements of $\Sigma=\set{x_1,\dots,x_d}$, such that from each vertex there is exactly one outgoing edge of each color. Then, any vertex $v\in V$ and a word $\xi\in\Sigma^*$ define a path in $G'$, by starting from $\xi$ and choosing the edges according to the letters of $\xi$. The random walk $R'_n$ then corresponds to sampling a random word $W_n=Y_1\cdots Y_n$ in $\Sigma^n$.

    For each $n\ge 1$, write $t_n=d^n$ and $k_n=n+\log_d n$, and define the event
    \[
        A_n=\set{\exists t:t-{n-1}+k_n\le t< t_n:\dist(o,R'_t)\ge\ff_{G'}(k_n)}.
    \]
    We will prove that $\Pr(A_n\mid R'_{d^{n-1}}=v)\gtrsim\frac{1}{n}$ for all $n\ge 1$ and $v\in V$.

    We define a function $\alpha_{n,v}\colon \Sigma^*\to\Sigma^{k_n}$ as follows: given $\xi\in\Sigma^*$, write by $w\in V$ the end vertex after starting a random walk from $v$ and following the steps of $\xi$. By the definition of $\ff_{G'}$, there is a path $w=w_0,\dots,w_k$ for $k\le k_n$ such that $\dist(w,w_k)=\ff_{G'}(k_n)$. We then define $\alpha_{n,v}(w)$ to be some word of length $k_n$, such that its first $k$ letters correspond to the path $w_0,\dots,w_k$.

    By \Pref{prop:subwords},
    \[
        \Pr(\exists t_{n-1}+k_n\le t<t_n:Y_{t-k_n+1}\cdots Y_t=\alpha_{n,v}(Y_{t_{n-1}+1},\dots,Y_{t-k_n})\mid R'_{t_{n-1}}=v) \gtrsim \frac{d^n}{d^{k_n}}=\frac{1}{n}.
    \]
    We note that if $R'_{t_{n-1}}=v$ and $Y_{t-k_n+1}\cdots Y_t=\alpha_{n,v}(Y_{t_{n-1}+1},\dots,Y_{t-k_n})$ for some $t_{n-1}+k_n\le t<t_n$, then by definition of $\alpha_{n,v}$ we have $\dist(o,R'_t)\ge\ff_{G'}(k_n)$; therefore we also have
    \[
        \Pr(A_n\mid R'_{t_{n-1}}=v) \gtrsim \frac{1}{n}.
    \]
    for all $v\in V$.

    The events $A_1,A_2,\dots$ are not independent, so we cannot apply the Borel-Cantelli lemma directly; however, the above is enough to follow the proof of the lemma. Indeed, we first note that by the strong Markov property for the random walk $R'_n$,
    \[
        \Pr(A_n^c\mid A_N^c,\dots,A_{n-1}^c)\lesssim 1-\frac{1}{n}
    \]
    for all $1\le N\le n$. This holds because $A_N^c,\dots,A_{n-1}^c$ depend only on $R'_{d^{N-1}},\dots,R'_{d^{n-1}}$, while we proved that $\Pr(A_n^c\mid R'_{d^{n-1}}=v)\gtrsim 1-\frac{1}{n}$ regardless of the value of $v\in V$.
    Next, for any $N\ge 1$ we have
    \[
        \Pr\left(\bigcap_{n=N}^{\infty}A_n^c\right) = \prod_{n=N}^{\infty}\Pr(A_n^c\mid A_N^c,\dots,A_{n-1}^c) \lesssim \prod_{n=N}^{\infty}\left(1-\frac{1}{n}\right)=0,
    \]
    where the last equality holds since $\sum_{n=N}^{\infty}\frac{1}{n}=\infty$. Therefore
    \begin{align*}
        \Pr((\limsup A_n)^c) = \Pr\left(\bigcup_{N=1}^{\infty}\bigcap_{n=N}^{\infty}A_n^c\right) = \lim_{N\to\infty}\Pr\left(\bigcap_{n=N}^{\infty}A_n^c\right) = 0.
    \end{align*}
    This shows that almost surely, the events $\set{A_n}$ hold infinitely often. Therefore, almost surely, there are infinitely many values of $t$ such that $\dist(o,R'_t)\ge\ff_G(\log_d t+\log_d\log_d t)$, which proves the same for $R_n$ as explained at the beginning of the proof.
\end{proof}

As random walks on semigroups are a special case of random walks on graphs, we first rephrase \Tref{thm:low-bound-io} for random walks on semigroups. For a finitely generated semigroup $S$, we denote by $\ff_S$ the function $\ff$ of its Cayley graph.

\begin{cor}\label{cor:lil-semigroups}
    Let $S$ be a $d$-generated semigroup with no finite right ideals, and let $R_n$ be a simple random walk on $S$. Then, almost surely, $\left|R_n\right|\ge\ff_S(\log_d n+\log_d\log_d n)$ for infinitely many values of $n$.
\end{cor}

While this bound is clearly not optimal for most semigroups, there are examples for which this is tight.

\begin{exmpl}\label{exmpl:speed-bounded-dist-not}
    Consider the semigroup $S=\sg{x,y\mid xy=y^2=y}$, and let $R_n$ be a simple random walk on $S$. We first note that by the diamond lemma, any element of $S$ can be written uniquely either as $x^m$ or as $yx^m$ for some integer $m\ge 0$. Therefore $\ff_S(n,a)=\left|a\right|+n$ for all $a\in S$, since $\left|ax^n\right|=\left|a\right|+n$, and thus $\ff_S(n)=n$ for all $n\in\N$. \Cref{cor:lil-semigroups} shows that, almost surely, $\left|R_n\right|\ge\log_d n+\log_d\log_d n$ for infinitely many values of $n$. We will show that this is optimal for the random walk $R_n$ on $S$.

    Consider the random variables $\set{\left|R_n\right|}_{n=0}^{\infty}$. We have $R_0=0$, $R_1=1$, and for all $n\ge 1$ we have $\Pr(R_{n+1}=R_n+1)=\Pr(R_{n+1}=1)=\frac{1}{2}$. We can therefore describe the distribution of $\left|R_n\right|$ as follows: given an infinite sequence of independent tosses of a fair coin, $\left|R_n\right|$ denotes the length of the current run of the tosses. By \cite{ErdosRevesz75}, for any $\eps>0$ we have $\left|R_n\right|\le\log_2 n+(1+\eps)\log_2\log_2 n$ for all but finitely many values of $n$.
\end{exmpl}

\section{Bounded speed}\label{sec8}

We finish the paper by considering the case where $\E\left|R_n\right|$ is bounded independently of $n$. This is the case, of course, if $S$ is finite, since then $\left|R_n\right|$ is itself bounded. To study this case, we utilize ideas from \cite{HognasMukherjea11}.

\begin{prop}\label{prop:finite-ideals}
    Let $S$ be a finitely generated semigroup, and let $R_n$ be a simple random walk on $S$. If $S$ contains a finite one-sided ideal, then $\E\left|R_n\right|$ is bounded.
\end{prop}

\begin{proof}
    Write $R_n=X_1\cdots X_n$, where $X_1,X_2,\dots$ are i.i.d.\ random variables distributed uniformly on the generating set of $S$. We prove the claim for right ideals; the claim for left ideals follows from a similar argument.

    Let $I\le S$ be a finite right ideal of $S$, i.e., $IS=I$, and write $M=\max_{a\in I}\left|a\right|$. Fix some $a_0\in I$, and write $m=\left|a_0\right|$. Define the stopping time
    \[
        \tau=\inf\set{n\mid X_{n-m+1}\cdots X_n=a_0}.
    \]
    It is a classical fact that $\E[\tau]<\infty$. We include a short proof here for completeness. We note that $\Pr(\tau>n)\le\Pr\left(\forall j\le\frac{n}{m}:X_{m(j-1)+1}\cdots X_{mj}=a\right)\le\left(1-\frac{1}{d^m}\right)^{\floor{\frac{n}{m}}}$. Then
    \[
        \E\left|R_n\right| = \E\left[\left|R_n\right|\mid \tau\le n\right]\Pr(\tau\le n) + \E\left[\left|R_n\right|\mid \tau>n\right]\Pr(\tau>n) \le M+n\left(1-\frac{1}{d^m}\right)^{\floor{\frac{n}{m}}}
    \]
    is bounded, since the last term tends to $0$ as $n$ tends to $\infty$.

    This proves that if $S$ contains a finite right ideal, then $\E\left|R_n\right|$ is bounded, completing the proof.
\end{proof}

\begin{rem}
    If $S$ contains a finite right ideal $I$, then $R_n$ is almost surely trapped in a finite set. Indeed, the proof of \Pref{prop:finite-ideals} shows that $\Pr(\exists n:R_n\in I)=1$, and since $I$ is a right ideal $R_n\in I$ implies $R_m\in I$ for all $m\ge n$. However, this is not the case when $S$ contains a finite left ideal but no finite right ideals, as demonstrated in \Eref{exmpl:speed-bounded-dist-not2}.
\end{rem}

In some cases, \Pref{prop:finite-ideals} gives a complete characterization of the cases where the speed is bounded, as demonstrated in the following proposition. Recall that an inverse semigroup is a semigroup in which for every element $x$ there exists a unique element $y$ (`weak inverse') such that $xyx=x,yxy=y$.

\begin{prop}
    Let $S$ be a finitely generated commutative semigroup or an inverse semigroup, and let $R_n$ be a simple random walk on $S$. If $\E\left|R_n\right|$ is bounded, then $S$ contains a finite two-sided ideal.
\end{prop}

\begin{proof}
    Suppose that $\E\left|R_n\right|\le M$ for all $n$. By Markov's inequality,
    \[
        \Pr(\left|R_n\right|\le 2M)\ge\frac{1}{2},
    \]
    for all $n$, and thus there exists $a\in B_S(1,2M)$ such that $\limsup_{n\to\infty}\Pr(R_n=a)>0$. This shows that the random walk $R_n$ is \textit{positive recurrent}. By \cite[Corollary 3.2]{HognasMukherjea11}, the Ces\`{a}ro sums $\frac{1}{n}\sum_{j=1}^n\mu^j$ converge to a probability measure $\pi$ supported on a completely simple minimal ideal $K$ of $S$ with a finite group factor in its Rees representation. In other words, $K=E\times G\times F$ where $G$ is a finite group, $E$ and $F$ are semigroups, and the multiplication is given by
    \[
        (e,g,f)(e',g',f') = (e,g\phi(f,e')g',f')
    \]
    for some function $\phi\colon F\times E\to G$. But if $S$ is either a commutative semigroup or an inverse semigroup, we must have $E=F=\set{1}$, so $K=G$ is a finite ideal of $S$.

    Indeed, if $S$ is commutative then for any $e,e'\in E,f,f'\in F$ we have $(e,\phi(f,e'),f')=(e,1,f)(e',1,f')=(e',1,f')(e,1,f)=(e',\phi(f',e),f)$ so $E,F$ are trivial.
    Furthermore, given $e\in E,f\in F$, for every $e'\in E,f'\in F$ we have that $(e',\phi(f,e')^{-1}\phi(f',e)^{-1},f')$ is a weak inverse of $(e,1,f)$, so $S$ cannot be an inverse semigroup unless $E,F$ are trivial.
\end{proof}

We conclude by giving an example of a semigroup for which the speed is bounded, while the random walk itself is not trapped in a finite subset almost surely.

\begin{exmpl}\label{exmpl:speed-bounded-dist-not2}
    Consider the semigroup $S=\sg{x,y\mid xy=y^2=y}$ of \Eref{exmpl:speed-bounded-dist-not}, and let $R_n$ be a simple random walk on $S$. Note that $\E\left|R_n\right|$ is bounded, since $y$ generates a finite left ideal $Sy=\set{y}$. However, as we saw before, the distance $\left|R_n\right|$ itself is almost surely unbounded.
\end{exmpl}

\appendix

\section{Appearances of subwords}\label{appendix}

The study of the number of appearances of subwords in a random word is a well-studied subject. In this section we consider the appearance of subwords in a random word, where the subwords change with the position and depend on the prefix of the word.

\begin{prop}\label{prop:subwords}
  Let $\Sigma=\set{x_1,\dots,x_d}$ be a finite alphabet, and let $X=X_1\cdots X_n\in\Sigma^n$ be a random word of length $n$ in $\Sigma$. Fix $k\le\frac{n}{2}$, and let $g\colon \Sigma^*\to\Sigma^k$ be a function. Then
  \[
    \Pr(\exists 1\le j\le n-k:X_{j+1}\cdots X_{j+k}=g(X_1\cdots X_j)) \ge \frac{1}{4(\frac{d^k}{n}+1)}.
  \]
\end{prop}

\begin{proof}
  Let $N=\sum_{j=1}^{n-k}I_j$, where $I_j=\ind{\set{X_{j+1}\cdots X_{j+k}=g(X_1\cdots X_j)}}$. We will compute the first two moments of $N$, and apply the Paley--Zigmund inequality to achieve the desired lower bound.

  To compute $\E[N]$, note that
  \[
    \E[I_j] = \Pr(X_{j+1}\cdots X_{j+k}=g(X_1\cdots X_j)) = \frac{1}{d^k},
  \]
  and thus
  \[
    \E[N] = \sum_{j=1}^{n-k}\E[I_j] = \frac{n-k}{d^k} \ge \frac{n}{2d^k}.
  \]
  For the second moment, let $1\le j<j'\le n-k$. If $j'-j\ge k$, then
  \[
    \E[I_jI_{j'}] = \Pr(X_{j+1}\cdots X_{j+k}=g(X_1\cdots X_j),X_{j'+1}\cdots X_{j'+k}=g(X_1\cdots X_{j'}))=\frac{1}{d^{2k}},
  \]
  since the equalities are independent. If $j'-j<k$, there is overlap between $X_{j+1}\cdots X_{j+k}$ and $X_{j'+1}\cdots X_{j'+k}$ of size $k-(j'-j)$, and thus
  \[
    \E[I_jI_{j'}] = \Pr(X_{j+1}\cdots X_{j+k}=g(X_1\cdots X_j),X_{j'+1}\cdots X_{j'+k}=g(X_1\cdots X_{j'})) \le \frac{1}{d^{2k-(j'-j)}}
  \]
  (where the last inequality is an equality if and only if $g(X_1\cdots X_j)$ and $g(X_1\cdots X_{j'})$ agree on the overlap, otherwise the left hand side is $0$). Therefore
  \begin{align*}
    \E[N^2] & = \sum_{j=1}^{n-k}\E[I_j] + 2\sum_{1\le j<j'\le n-k}\E[I_jI_{j'}] \\
    & = \frac{n-k}{d^k} + 2\sum_{r=1}^{k-1}\sum_{j=1}^{n-k-r}\E[I_jI_{j+r}] + 2\sum_{j'-j\ge k}\E[I_jI_{j'}] \\
    & \le \frac{n-k}{d^k} + 2\sum_{r=1}^{k-1}\frac{n-k}{d^{2k-r}} + \frac{(n-k)^2}{d^{2k}} \\
    & \le \frac{n-k}{d^k} + \frac{2}{d^k(d-1)} + \frac{(n-k)^2}{d^{2k}} \le \frac{n}{d^k} + \frac{n^2}{d^{2k}}.
  \end{align*}

  By the Paley--Zygmund inequality,
  \[
    \Pr(N>0) \ge \frac{\E[N]^2}{\E[N^2]} \ge \frac{\frac{n^2}{4d^{2k}}}{\frac{n}{d^k}+\frac{n^2}{d^{2k}}} \ge \frac{1}{4(\frac{d^k}{n}+1)}
  \]
  completing the proof.
\end{proof}

\bibliographystyle{plain}
\bibliography{refs}

\end{document}